\RequirePackage[l2tabu, orthodox]{nag}

\documentclass[12pt]{amsart}
\usepackage{fullpage,url,amssymb,enumerate}
\usepackage[all]{xy} % for commutative diagrams
\usepackage{comment}
\usepackage{graphicx}

\usepackage[OT2,T1]{fontenc}

\usepackage{amsthm}

\usepackage{color}
\def\cC{\mathcal{C}}

\def\cM{\mathcal{M}}

\def\cO{\mathcal{O}}
\def\bP{\mathbb{P}}

\def\bQ{\mathbb{Q}}

\def\bZ{\mathbb{Z}}

\def\barM{\overline{\cM}}

\DeclareMathOperator{\ev}{ev}

\DeclareMathOperator{\Pic}{Pic}

\DeclareMathOperator{\Spec}{Spec}

\newtheorem{thm}{Theorem}
\newtheorem{lem}[thm]{Lemma}
\newtheorem{cor}[thm]{Corollary}
\newtheorem{prop}[thm]{Proposition}

\newtheorem{speculation}[thm]{Speculation}
\newtheorem{question}[thm]{Question}

\newtheorem{rem}[thm]{Remark}
\newtheorem{ex}[thm]{Example}
\newtheorem{defn}[thm]{Definition}

\newcommand{\Tev}{{\mathsf{Tev}}}
\newcommand{\vTev}{{\mathsf{vTev}}}

\makeatletter
\g@addto@macro\bfseries{\boldmath} % This makes math in section titles bold.
\makeatother

\usepackage{microtype}

\usepackage[
	backref,
	pdfauthor={}, % add other authors
	pdftitle={tevelev_asymptotic},
]{hyperref}

\begin{document}

\title{Enumerativity of virtual Tevelev degrees}

\author{Carl Lian}

\address{Humboldt-Universit\"at zu Berlin, Institut f\"ur Mathematik,  Unter den Linden 6
\hfill \newline\texttt{}
 \indent 10099 Berlin, Germany} \email{{\tt liancarl@hu-berlin.de}}

\author{Rahul Pandharipande}

\address{ETH Z\"urich, Departement Mathematik,  R\"amisstrasse 101
\hfill \newline\texttt{}
 \indent 8044 Z\"urich, Switzerland} \email{{\tt rahul@math.ethz.ch}}
\date{March 2023}

\begin{abstract}
  Tevelev degrees in Gromov-Witten theory are defined whenever there are virtually a
  finite
  number of genus $g$ maps of fixed complex structure in a given curve class $\beta$
  through $n$ general points
  of a target variety $X$. These virtual Tevelev degrees often have much simpler structure than
  general Gromov-Witten invariants. We explore here the question of
  the enumerativity of such counts in the asymptotic range for large
  curve class $\beta$. A simple speculation is that for all Fano $X$, the virtual
  Tevelev degrees are enumerative for sufficiently large $\beta$. We prove the claim
  for all homogeneous varieties and all hypersurfaces of sufficiently low degree (compared
  to dimension). As an application, we prove a new result on the existence of very free curves of low degree on hypersurfaces in positive characteristic.
\end{abstract}

\maketitle

\setcounter{tocdepth}{1}
\tableofcontents

\section{Tevelev degrees}
\subsection{Definitions}
Let $X$ be a nonsingular, projective, complex algebraic variety
of dimension $r$,  and let $\beta\in H_2(X,\bZ)$ be a class satisfying
\begin{equation}\label{ppp}
  \int_\beta c_1(T_X) > 0\, .
  \end{equation}
  If $X$ is Fano, the positivity \eqref{ppp} is always satisfied
  for classes of curves. Fano varieties will
  be our main interest here.

  Let $g\geq 0$ and $n\geq 0$ be in the stable range $2g-2+n>0$ so that the moduli space of stable
  curves $\barM_{g,n}$ is well-defined.
  The moduli space of stable maps
  $\barM_{g,n}(X,\beta)$ has virtual dimension equal to the dimension of $\barM_{g,n}\times X^n$
  if and only if
  \begin{equation}\label{ddd}
    \int_\beta c_1(T_X)=r(n+g-1)\, .
  \end{equation}
  If the dimension constraint \eqref{ddd} holds, we expect to find a 
finite number  of maps  from a fixed curve
$(C,p_1,\ldots,p_n)$ of genus $g$  to $X$  of curve class $\beta$
  where the $p_i$ are incident to general points of $X$.
  Tevelev degrees in Gromov-Witten theory are defined to be the
  corresponding virtual count.

\begin{defn}
Let $g\geq 0$, $n\geq 0$, and $\beta\in H_2(X,\mathbb{Z})$ satisfy $2g-2+n>0$ and the dimension constraint \eqref{ddd}.
  Let $$\tau:\barM_{g,n}(X,\beta)\to\barM_{g,n}\times X^n$$ be the canonical morphism.
The \textbf{virtual Tevelev degree} $\vTev^{X}_{g,\beta,n}\in\bQ$ is defined by
$$\tau_{*}([\barM_{g,n}(X,\beta)]^{vir}) = \vTev^{X}_{g,\beta,n} \cdot [\barM_{g,n}\times X^n] \in A^0(\barM_{g,n}\times X^n)\, .$$
\end{defn}
\noindent Here, $[\,]^{vir}$ and $[\,]$ denote the virtual and usual fundamental classes, respectively.
%If either $2g-2+n\leq 0$ or 
If the dimension constraint \eqref{ddd} fails, we define $\vTev^{X}_{g,\beta,n}$ to vanish.

We say that the virtual Tevelev count for $g$ and $\beta$ is \textbf{well-posed} if
$$n(g,\beta)= 1-g +\frac{1}{r}\int_\beta c_1(T_X)\, $$
is a non-negative integer satisfying
$2g-2+n(g,\beta)>0$.
%and the dimension constraint \eqref{ddd}. 
We will use the notation
$$\vTev^{X}_{g,\beta}= \vTev^{X}_{g,\beta,n(g,\beta)}\,   $$
in the well-posed case.

Let $(C,p_1,\ldots,p_n)$ be a fixed {\em general} nonsingular curve of genus $g$ with $n$ distinct points.
Let $x_1,\ldots, x_n\in X$ be $n$ general points. 

\begin{defn}
If the virtual Tevelev count for $g$ and $\beta$ is well-posed and
the actual count of maps 
$$f: (C,p_1,\ldots,p_n) \rightarrow X$$
  in class $\beta$ satisfying $f(p_i)=x_i$ is finite
  and transverse, we define the \textbf{geometric Tevelev degree} $\Tev^{X}_{g,\beta,n}\in\bZ$ to be the set-theoretic count.
  \end{defn}

Equivalently, $\Tev^{X}_{g,\beta,n}$ is defined by the set-theoretic degree of the morphism 
$$\tau_\cM:\cM_{g,n}(X,\beta)\to\cM_{g,n}\times X^n\ $$
computed along a general fibre (which is required to be everywhere
transverse).
Transversality here is the condition that $d\tau_\cM$ is
an isomorphism on Zariski tangent spaces.
%where we have restricted to the loci of nonsingular curves and stable maps with nonsingular domains (and, 
The definition ignores all components of $\cM_{g,n}(X,\beta)$ which fail to dominate $\cM_{g,n}\times X^n$.

The virtual Tevelev degree is  \textbf{enumerative} if the geometric
Tevelev degree is well-defined and 
$$\vTev^{X}_{g,\beta,n}=\Tev^{X}_{g,\beta,n}\, .$$

\begin{rem}\em{
The virtual Tevelev count is never well-posed for constant maps: if
$\int_\beta c_1(T_X)=0$
and  $n(g,0)\geq 0$, then $g$ must be 0 or 1 with
$$2g-2+n(g,0) < 0\, .$$
The geometric Tevelev degree therefore requires a nonzero
class $\beta$. } 

\end{rem}

\begin{rem}\em{
Stable maps with automorphisms {\em never} occur
in a general transverse fiber of $\tau_\cM$ when the
geometric Tevelev degree is well-defined. 
Such automorphisms could only occur in the cases
 \begin{equation}\label{cc12}
 (g,n)=(1,1)\  \text{or} \ (2,0)\, .
 \end{equation}
 In both cases \eqref{cc12}, a stable map to $X$ in a general fiber of $\tau_\cM$ with a nontrivial automorphism must
 factor through a map to $\bP^1$. Using the infinite automorphism group of $\bP^1$, we see the finiteness
 condition for the geometric Tevelev count is violated.}
\end{rem}

\subsection{Calculations of virtual Tevelev degrees}
  While Gromov-Witten invariants are in general rarely enumerative (especially
  in higher genus) and complicated to compute, the situation is much better for virtual Tevelev degrees.
  
  \vspace{8pt}
  \noindent $\bullet$ The projective space case $X=\bP^r$ has a particularly simple
  answer:
 \begin{equation} \label{pppp}
 \vTev^{\bP^r}_{g,d}=(r+1)^g
 \end{equation}
 whenever the virtual Tevelev count is well-posed \cite{Bertram94, BDW,
   SiebertTian97}.

  \vspace{8pt}
  \noindent $\bullet$ For low degree hypersurfaces, a similar result is true.
  \begin{thm}\cite{bp}\label{virtualtev_hypersurface} Let
    $X_e\subset \mathbb{P}^{r+1}$ be a nonsingular hypersurface of degree
    $e\geq 3$ and dimension $r\geq 2e-3$. We 
    index curve classes of $X_e$ 
    by their associated degree $d$ in $\mathbb{P}^{r+1}$.
    Then,
    $$\vTev^{X_e}_{g,d} =  ((e-1)!)^{n(g,d)} \cdot (r+2-e)^g\cdot  e^{(d-n(g,d))e-g+1}$$
%    $$\vTev^{X_e}_{g,d} =  ((e-1)!)^{n(g,d)} \cdot ((r+2-e) e^{e-1})^g\cdot  e^{(d-n(g,d)-g)e+1}$$
       whenever the virtual Tevelev count is well-posed.
   \end{thm}

   \vspace{8pt}
   \noindent$\bullet$ The $e=2$ case of quadric hypersurfaces
   $Q^r\subset \bP^{r+1}$ 
   takes
   a special form. Let
   $$\delta= \begin{cases} {\text{1 \ \ if $r$ is odd}}, \\
   {\text{2 \ \ if $r$ is even.}}
\end{cases}$$
\begin{thm}\cite{bp}\label{quadric} 
For quadrics of dimension $r\geq 3$,
    $$\vTev^{Q^r}_{g,d} =  \frac{(2r)^g+(-1)^{d}(2\delta)^g}{2}$$
    whenever the virtual Tevelev count is well-posed.
   \end{thm}

The method of \cite{bp} expresses the
virtual Tevelev degrees explicitly in terms of the
small quantum cohomology ring of $X$. When $QH^*(X)$ is
sufficiently well known, exact calculations are possible.
For further recent progress on the Gromov-Witten of hypersurfaces, see
\cite{ABPZ,Hu1}.

\subsection{Enumerativity}
  Our main topic here is the enumerativity of virtual Tevelev degrees. When $\beta$ is sufficiently large, enumerativity is
  much more likely as evidenced by the following example, where  $X=\bP^1$.
  
\begin{thm}\cite{cps,fl}\label{geomtev_Pr}
Suppose the virtual Tevelev count is well-posed. Then
\begin{align*}
    \Tev^{\bP^1}_{g,d}&=2^g-2\sum_{i=0}^{-\ell-2}\binom{g}{i}+(-\ell-2)\binom{g}{-\ell-1}+\ell\binom{g}{-\ell}\\
&=\int_{\text{Gr}(2,d+1)}\sigma_{1}^{g}\cdot\left[ \sum_{a+b=2d-2-g}\sigma_{a}\sigma_b\right],
\end{align*}
where $\ell=d-g-1$ in the first formula.
\end{thm}
\noindent When $\ell\le0$, all terms except $2^g$ in the first formula are interpreted to be zero. In particular, virtual Tevelev degrees for $\bP^1$ are enumerative if $d\ge g+1$, but not in general.
For further results related to moduli spaces of Hurwitz covers, see
\cite{CelaLian}.

More generally, virtual Tevelev degrees for $\bP^r$ are enumerative whenever $d\ge rg+r$ but not when $d=r+\frac{gr}{r+1}$ is as small as possible \cite{fl}. The geometric Tevelev degree in the case $d=r+\frac{gr}{r+1}$ recovers Castelnuovo's count of linear series of minimal degree \cite{castelnuovo}. 
The general computation of geometric Tevelev degrees for $\bP^r$ for intermediate $d$ is in completed in \cite{lian}. 

These current developments on Tevelev degrees for $\mathbb{P}^1$
and higher dimensional projective spaces 
may be viewed (in part) as developing a theory intiated by Castelnuovo
in the $19^{th}$ century. The question of connecting classical 
counting to  virtual counting can be formulated as follows.

\begin{question} \label{qqq}
  For which $X$ does the following property hold:
  $\vTev^{X}_{g,\beta}$ is enumerative whenever well-posed and
  $\int_\beta c_1(T_X)$ is sufficiently large (depending on $g$)?
\end{question}

\begin{rem} {\em The $2^g$ formula for $\bP^1$
is connected to many directions
in geometry and physics, see Tevelev's article \cite{Tev}. 
In the Gromov-Witten theory of $\bP^1$, the $2^g$ formula 
appeared in Janda's work \cite{Janda}.
The geometric Tevelev
degrees for $\bP^r$ for large curve classes were studied earlier by Bertram, Daskalopoulos, and Wentworth 
\cite{Bertram94,BDW}
using the classical geometry  of the Quot scheme
before the development
of the virtual fundamental class.
 To connect the Quot scheme fully (for all curve classes) to the
Gromov-Witten calculation \eqref{pppp}, a straightforward path is to
consider the virtual fundamental class of the
Quot scheme \cite{MO} and then  apply the comparison result
\cite[Theorem 3]{MOP}. Alternatively, formula \eqref{pppp} is
a direct consequence of the Vafa-Intriligator formula \cite{SiebertTian97}.}
\end{rem}

\subsection{Main results}
We have positive answers to Question \ref{qqq} for homogeneous spaces and hypersurfaces. Together with the
above calculations of $\vTev^{X}_{g,d}$ from \cite{bp}, we obtain
calculations of geometric Tevelev degrees in many new cases.

\begin{thm}\label{thm_g/p} Let $X=G/P$ be a homogeneous space for a linear algebraic group. Then, for fixed $g$, the virtual Tevelev degree
  $\vTev^{X}_{g,\beta}$ is enumerative whenever well-posed and
  $\int_\beta c_1(T_X)$ is sufficiently large.
\end{thm}

In case $g=0$, a stronger result holds for $X=G/P$: the virtual Tevelev degrees $\vTev^{X}_{0,\beta}$ are enumerative for \emph{all} positive curve classes $\beta$. The stronger genus $0$ claims follows easily from the
unobstructedness of genus $0$ stable maps to  $G/P$.
 In case $X$ is a Grassmannian, the enumerativity claim of
Theorem \ref{thm_g/p} is known from the results of \cite{Bertram94,BDW}
and the comparison results of \cite{MOP}.

  \begin{thm}\label{thm_hypersurface} Let $X\subset \mathbb{P}^{r+1}$ be a nonsingular hypersurface of degree
    $e\geq 3$ and dimension $$r>(e+1)(e-2)\, .$$ Then, for fixed $g$, the virtual Tevelev degree
    $\vTev^{X}_{g,\beta}$ is enumerative whenever well-posed and
    $\int_\beta c_1(T_X)$ is sufficiently large.
\end{thm}

In case $g=0$ and $r>(e+1)(e-2)$, we again have a stronger result:
the virtual Tevelev degrees $\vTev^{X}_{0,\beta}$ are enumerative for \emph{all} positive curve classes $\beta$, see Corollary \ref{hyp_g=0}. If $X$ is a \emph{very general} hypersurface,  the stronger claim
follows from the fact that $\overline{M}_{0,n}(X,\beta)$ is irreducible of the expected dimension for $r>e$  and
all positive curve classes $\beta$ \cite{bk,hrs,ry}. Our proof works for \emph{all} $X$ in the more restrictive range for $e$.

It is natural to hope for the following result which we formulate as a speculation.

\begin{speculation}\label{main_speculation} Let $X$ be a nonsingular projective Fano variety. For fixed $g$, the virtual Tevelev degree
  $\vTev^{X}_{g,\beta}$ is enumerative whenever well-posed and
  $\int_\beta c_1(T_X)$ is sufficiently large.
% and always if $g=0$.
\end{speculation}

\noindent In case $g=0$, $\vTev^{X}_{0,\beta}$ is enumerative
whenever well-posed for all positive curve classes $\beta$ in all Fano examples that we know the speculation to be true.{\footnote{Speculation \ref{main_speculation} was proposed in 2021. In 2023,
    counterexamples to Speculation \ref{main_speculation} have been constructed by Beheshti-Lehmann-Riedl-Starr-Tanimoto \cite{blrst} when $X$ is a special Fano hypersurface of large degree. We thank Eric Riedl for communicating these
    examples to us. A revised speculation should perhaps include a condition
  on the Fano index.}

As we will see, there are two main difficulties in proving a general
enumerativity statement for Fano varieties $X$:

\vspace{6pt}
\noindent $\bullet$ The first concerns controlling the excess dimensions of families of \textit{general} curves of positive genus in $X$. When $X=\bP^r$, Brill-Noether theory provides optimal statements (see Remark \ref{brill_noether}), but we do not have such results in general. 

\vspace{6pt}
\noindent $\bullet$ The second concerns controlling
the excess dimensions of families of rational curves. Similar issues for hypersurfaces are studied in \cite{bk,hrs,ry} in characteristic 0 and in
\cite{STZ} in positive characteristic.

\vspace{6pt}
Using our study of the enumerativity of virtual Tevelev
degrees of hypersurfaces, we can 
can prove a new result on the existence of very free rational curves in characterstic
$p$ (where $p$ does not divide the virtual Tevelev degree).
The positive characteristic results are presented in Section \ref{very_free_sec}.

\subsection{Acknowledgments}

Our project started during a visit to
HU Berlin by R.P. in July 2021 and
continued at the
{\em Helvetic Algebraic Geometry Seminar}
in Geneva 
and at the
{\em Forschungsinsitut f\"ur Mathematik} at
ETH Z\"urich
in August 2021. We thank A. Buch, A. Cela, G. Farkas, D. Ranganathan, E. Riedl, and J. Schmitt for many discussions about Tevelev
degrees in various contexts. The application to very free
rational curves on hypersurfaces in characteristic $p$ was
suggested to us by J. Starr.

C.L. was funded by an NSF postdoctoral
fellowship, grant DMS-2001976. He gratefully acknowledges the Institut f\"{u}r Mathematik at Humboldt-Universit\"{a}t zu Berlin for its continued support.
R.P. was supported by SNF-200020-182181,  ERC-2017-AdG-786580-MACI, and SwissMAP.  
The project has received funding
from the European Research Council (ERC) under the European Union Horizon 2020 research and innovation program (grant agreement No 786580).

\section{Analysis of the moduli space of stable maps}
\subsection{Overview}
Let $X$ be a nonsingular projective
variety of dimension $r$. Let  $g\geq 0$ be the genus, and let $\beta\in H_2(X,\mathbb{Z})$ be an effective
curve class.

We study here the enumerativity of virtual Tevelev degrees when $\beta$ (or, equivalently, $n$) is sufficiently large. There are two main aspects of the analysis: 

\begin{enumerate}
\item[(i)] We must control the fibers of $\tau$ when restricted to the open locus $$\cM_{g,n}(X,\beta)\subset \barM_{g,n}(X,\beta)$$ of stable maps with nonsingular domains in order to verify that the geometric Tevelev degrees are well-defined.
\item[(ii)] We must show that a general fiber of $\tau$ contains no stable maps at the boundary.
\end{enumerate}

\subsection{Stable maps with nonsingular domains} \label{smnd}
We start with a criterion for unobstructedness of maps of 
nonsingular curves $C$ to a nonsingular projective
variety $X$. We do not require $X$ to be Fano in
Section \ref{smnd}.

\begin{prop}\label{H1_vanishing}
%Let $(C,p_1,\ldots,p_n,x_1,\ldots,x_n)\in \cM_{g,n}\times X^n$ be a general point.
Suppose $[f:(C,p_1,\ldots,p_n)\to X]\in \cM_{g,n}(X,\beta)$ lies over a point $$(C,p_1,\ldots,p_n,x_1,\ldots,x_n)\in \cM_{g,n}\times X^n$$ and  the 
evaluation
map $$\ev:\cM_{g,n}(X,\beta)\to X^n$$ is surjective on tangent spaces at $[f]$. Assume further that $(C,p_1,\ldots,p_n)\in\cM_{g,n}$ is general.
If $n\ge g+1$, or equivalently $\int_\beta c_1(T_X)\ge 2gr$,
then $H^1(C,f^{*}T_X)=0$.
\end{prop}

\begin{proof}
Let $v$ be a non-zero tangent vector of $X$ at $x_1$. By assumption, there exists a tangent vector of $\cM_{g,n}(X,\beta)$ at $[f]$ mapping to $(v,0,\ldots,0)\in T_{(x_1,\ldots,x_n)}X^n$. This is equivalent to the data of a section $\phi_v:\cO_C(\sum_{i=2}^{n}p_i)\to f^{*}T_X$ evaluating to $v\in T_{x_1}X$ at $p_1$.

Varying over a basis of $T_{x_1}X$, we obtain a map of vector bundles on $C$, $$\phi:\cO_C\left(\sum_{i=2}^{n}p_i\right)^{\oplus r}\to f^{*}T_X\, ,$$ that is surjective at $p_1$ and therefore generically surjective. Thus, the induced map on $H^1$ is surjective. 

On the other hand, we claim that $H^1(\cO_C(\sum_{i=2}^{n}p_i))=0$, which implies the needed conclusion. The $H^1$-vanishing is an open condition, so it suffices to degenerate to the situation in which the $p_i$ become equal to a single general point $p$. If $h^1(\cO_C((n-1)p))>0$, then $C$ is a general curve possessing a linear series $V$ of degree $d=n-1\ge g$ and rank $s\ge d-g+1$, ramified to order $d-1$ at $p$. The pointed Brill-Noether number of $V$ is 
\begin{equation*}
\widehat{\rho}=g-(s+1)(g-d+s)-(d-1)=-s(s-(d-g)+1)+1<0,
\end{equation*}
as $d\ge g$ and $s\ge 1$, contradicting the pointed Brill-Noether theorem \cite[Proposition 1.2]{eh}.
\end{proof}

If $n\ge 2g$, we also obtain the conclusion for \emph{arbitrary} pointed curves $(C,p_1,\ldots,p_n)$, as $H^1(\cO_C(\sum_{i=2}^{n}p_i))=0$ for degree reasons. When $g=0$, we must assume further that $n\ge1$ in order to choose the point $x_1$.

Suppose $\vTev^X_{g,\beta}$ is well posed, and let $n=n(g,\beta)$. Let $$(C,p_1,\ldots,p_n,x_1,\ldots,x_n)\in \cM_{g,n}\times X^n$$ be a general point. 

\begin{prop}\label{transversality_smooth}
If $n\ge g+1$, then there are finitely many maps $$[f:(C,p_1,\ldots,p_n)\to X]\in \cM_{g,n}(X,\beta)$$ lying over $(C,p_1,\ldots,p_n,x_1,\ldots,x_n)\in \barM_{g,n}\times X^n$, 
%(where we require the source of the stable map to be isomorphic to $C$) 
and all such maps are transverse. 
\end{prop}

\begin{proof}
Suppose $[f:(C,p_1,\ldots,p_n)\to X]\in \cM_{g,n}(X,\beta)$ is such a map. Then, $[f]$ must lie on a component of $Z\subset \cM_{g,n}(X,\beta)$ dominating $\cM_{g,n}\times X^n$. In particular, the evaluation map $$\text{ev}: \cM_{g,n}(X,\beta)\to X^n$$ is dominant on $Z$, so the map $\text{ev}$ is surjective on tangent spaces at $[f]$ (as the $x_i$ are general).

By Proposition \ref{H1_vanishing}, $\cM_{g,n}(X,\beta)$ is nonsingular of the expected dimension at $[f]$.
Therefore, $Z$ is \'{e}tale over $\cM_{g,n}\times X^n$ at $[f]$. Finiteness and transversality then follow.
\end{proof}

By Proposition \ref{transversality_smooth}, in the
well-posed case with $n\ge2g$, the geometric Tevelev degree is well-defined  and equal to the degree of $$\tau:\cM_{g,n}(X,\beta)\to\cM_{g,n}\times X^n\, .$$

\begin{rem}\label{brill_noether}
{\em The bound $n\ge g+1$ is not sharp. For example, if $X=\bP^r$, then by the Brill-Noether Theorem, geometric Tevelev degrees are also well-defined whenever $n\ge r+1$, in which case $f:C\to\bP^r$ is necessarily non-degenerate.}
\end{rem}

%The following is a weaker variant of Proposition \ref{H1_vanishing} that we will also need.
%
%\begin{prop}\label{H1_constant}
%Let $(C,p_1,\ldots,p_n,x_1,\ldots,x_n)\in \cM_{g,n}\times X^n$ be a general point.
%
%Suppose $[f:(C,p_1,\ldots,p_n)\to X]\in \cM_{g,n}(X,\beta)$ lies over $(C,p_1,\ldots,p_n,x_1,\ldots,x_n)\in \barM_{g,n}\times X^n$ for which the map $\ev_1:\cM_{g,n}(X,\beta)\to X$ is surjective on tangent spaces at $[f]$.
%
%Then, $h^1(C,f^{*}T_X)\le gr$. In particular, $h^1(C,f^{*}T_X)$ is bounded above by a constant not depending on $\beta$.
%\end{prop}

%\begin{proof}
%Let $v$ be a non-zero tangent vector of $X$ at $x_1$. Then, there exists a tangent vector of $\cM_{g,n}(X,d\beta)$ at $[f]$ mapping to $v\in T_{x_1}X$. As in the proof of Proposition \ref{H1_vanishing}, we obtain a generically surjective map of vector bundles $\phi:\cO_C^{\oplus r}\to f^{*}T_X$. (Unlike in Proposition \ref{H1_vanishing}, we are unable to require that $\phi$ vanish at $p_2,\ldots,p_n$.) We therefore have a surjection $H^1(C,\cO_C)^{\oplus r}\to H^1(C,f^{*}T_X)$, yielding the conclusion.
%\end{proof}

\subsection{Stable maps with singular domains}

We now assume $X$ is a projective Fano variety of dimension $r$.

\begin{lem}\label{unirational_expdim}
For maps from $\bP^1$ to $X$, we have the
following basic results:
\begin{enumerate}
\item[(a)] If $\ev_1:\cM_{0,2}(X,\beta)\to X$ is dominant (on every component of the source), then $\ev_2:\cM_{0,2}(X,\beta)\to X$ is also dominant, and $\cM_{0,2}(X,\beta)$ is generically nonsingular of the expected dimension.
\item[(b)] If $f^{*}T_X$ is globally generated for every $f:\bP^1\to X$, then every boundary stratum of $\overline{M}_{0,n}(X,\beta)$ is nonsingular of the expected dimension.
\end{enumerate}
\end{lem}

\begin{proof}
First, consider claim (a). At a generic
point $[f:\bP^1\to X]$ of any irreducible component of $\cM_{0,2}(X,\beta)$ on which $\ev_1$ is dominant, all summands of $T_X|_{\bP^1}$ are non-negative, so $\ev_2$ is also dominant on that component. Moreover, $H^1(\bP^1,T_X)=0$ at $[f]$, so $\cM_{0,2}(X,\beta)$ is generically nonsingular of the expected dimension.

Claim (b) follows from the same argument as (a), by induction on the number of components of the domain curve in the stratum in question. See, for example,
\cite{FultonP}.
\end{proof}

%\begin{lem}
%Let $(C,p_1,\ldots,p_n,x_1,\ldots,x_n)\in \barM_{g,n}\times X^n$ be a general point.
%
%Let $[f:(C',p_1,\ldots,p_n)\to X]\in \barM_{g,n}(X,d\beta)$ be a point lying over $(C,p_1,\ldots,p_n,x_1,\ldots,x_n)\in \barM_{g,n}\times X^n$, and suppose that $C'$ is reducible.
%
%Then, $C'$ is the union of a genus $g$ component $C_0\cong C$, along with a collection of disjoint rational trees attached to $C'$, each of whose components maps to $X$ by a non-constant map, and each of which contains at most one of the $p_i$.
%
%Furthermore, the number of $p_i$ lying on $C'$ is strictly less than $2g$.
%\end{lem}
%
%\begin{proof}
%The arguments of \ref{H1_vanishing} and \ref{transversality_smooth} still work, because Fano.
%\end{proof}

\begin{defn}
Let $s(X)>0$ be the smallest positive integer for which there exists an effective curve class $\beta\in H_2(X,\bZ)$ such that $$s(X)=\int_\beta c_1(T_X)$$ 
and $\text{ev}_1:\overline{M}_{0,1}(X,\beta)\to X$ is surjective.

Let $t(X)>0$ be the smallest positive integer for which there exists an effective curve class $\beta\in H_2(X,\bZ)$ such that $$t(X)=\int_\beta c_1(T_X)\, .$$
\end{defn}

\begin{defn}
We say  $X$ has property $(\star)_g$ if, for 
 every curve class $\beta$ and for every $[f:C\to X]\in\cM_g(X,\beta)$, we have
\begin{equation*}
h^1(C,f^{*}T_X)\le K_{g,X},
%h^1(C,f^{*}T_X)\le O(1),
\end{equation*}
for some constant $K_{g,X}$
depending only on $g$ and $X$ (so {\em not} on $C$, $\beta$, and $f$).

We say that $X$ has property $(\star\star)_g$ if 
\begin{equation*}
\liminf \frac{\int_\beta  c_1(T_X)}{h^1(C,f^{*}T_X)}>r-s(X)\, .
\end{equation*}
where we range over all curve classes $\beta$ and $[f:C\to X]\in\cM_g(X,\beta)$ and order by $\int_X\beta\cdot c_1(T_X)$. 
\end{defn}
\noindent Property $(\star\star)_g$ is automatically satisfied if $(\star)_g$ is, or if $s(X)>r$.

\begin{ex}
{\em If $X=G/P$, then $T_X$ is globally generated, so
\begin{equation*}
h^1(C,f^{*}T_X)\le gr,
\end{equation*}
In particular, $X$ satisfies $(\star)_g$ for any $g$.}
\end{ex}

\begin{ex}
We will see in later in Proposition \ref{complete_intersection} and the proof of Theorem \ref{thm_hypersurface} that hypersurfaces $X_e\subset\bP^{r+1}$ of sufficiently low degree $e$ satisfy $(\star\star)_g$.
\end{ex}

% \begin{ex}
% {\em If $X$ is a blowup of $\bP^r$, then $s(X)=r+1$, so $X$ satisfies $(\star\star)_g$.
% In particular, every del Pezzo surface satisfies $(\star\star)_g$.}
% \end{ex}

%\begin{ex}
%{\em Suppose $X$ is a blowup of $(\bP^1)^r$ at a point for $r\ge2$. Then, $s(X)=2$, so $$r-s(X)\le0\, .$$
%Let $C$ be a smooth curve of degree $d$ in the exceptional divisor of $X$. We find $\int_{C}c_1(T_X)=d$, but $h^1(C,f^{*}T_X)=O(d^2)$, so $X$ fails to satisfy $(\star\star)_g$.
%
%More generally, if $Y$ satisfies $s(Y)\le\dim(Y)$ and $X$ is the blowup of $Y$ at a point, then $X$ fails to satisfy  $(\star\star)_g$.}
%\end{ex}

We are now ready to show that, under certain hypotheses, stable maps $[f:C\to X]$ at the boundary of the moduli space
cannot contribute to the virtual Tevelev degree if $n$ is sufficiently large. We first consider the case in which $C$ is the union of a nonsingular component and disjoint nonsingular rational tails, each containing a marked point $p_i$.

\begin{prop}\label{rational_tails_simple}
Suppose $X$ satisfies property $(\star\star)_g$ and $n$ is sufficiently large (depending on $X$ and $g$). Let 
$$\cM_{\Gamma}\subset \barM_{g,n}(X,\beta)$$ be a locally closed boundary stratum consisting of stable maps $f:C\to X$ such that $C$ is the union of a nonsingular genus $g$ curve $C_0$ and disjoint nonsingular rational tails $R_1,\ldots,R_{n-m}$, such that $x_i\in R_{i}$ for $i=1,2,\ldots,n-m$ and $x_i\in C_0$ for $i=n-m+1,\ldots,n$. (See Figure \ref{Fig:domain_tails}.)

Suppose further that the virtual dimension of $\barM_{g,n}(X,\beta)$ is \textit{at most} the dimension of $\barM_{g,n}\times X^n$, 
\begin{equation}\label{dim_constraint_weak}
r(1-g)+\int_\beta  c_1(T_X)\le rn\, ,
\end{equation}
and furthermore that, if equality holds, then $n-m>0$.

Then, $\dim(\cM_\Gamma)<\dim(\barM_{g,n}\times X^n)$. In particular, $\cM_\Gamma$ fails to dominate $\barM_{g,n}\times X^n$.
\end{prop}

\begin{figure}[!htb]
     \includegraphics[width=.75\linewidth]{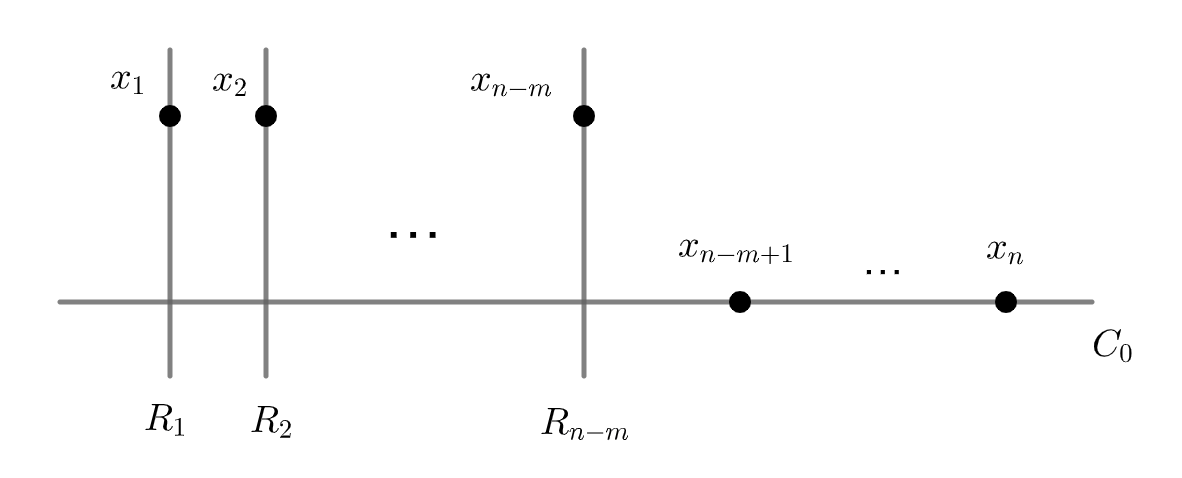}
          \caption{Domain of a stable map given by a nonsingular $C_0$ attached to disjoint nonsingular rational tails}
          \label{Fig:domain_tails}
\end{figure}

\begin{proof}
If $n-m=0$ and $n\ge g+1$, then $\cM_\Gamma=\cM_{g,n}(X,\beta)$ has expected dimension by Proposition \ref{H1_vanishing}, so we obtain the claim.

If $m\ge g+1$ and $n-m>0$, then $\cM_{g,n}(X,f_{*}[C_0])$ has expected dimension by the argument of Proposition \ref{H1_vanishing}. Moreover, $\cM_{0,2}(X,f_{*}[R_i])$ has expected dimension by Lemma \ref{unirational_expdim}(a), and any incidence condition on the nodal point of $R_i$ imposes the expected number of conditions. Thus, the boundary stratum $\cM_\Gamma$ has expected dimension, which is strictly less than that of $\cM_{g,n}(X,\beta)$, so we are again done.

Assume now that $m<g+1$, so in particular $m$ is bounded above by a constant. Applying Lemma \ref{unirational_expdim}(a) again, we find
\begin{equation*}
\dim(\cM_{\Gamma})\le\dim(\barM_{g,n}\times X^n)+h^1(C_0,T_X)-n+O(1)\, .
\end{equation*}
where the term $O(1)$ denotes a constant upper bound depending only on $g$ and $X$.

Furthermore, we have
\begin{align*}
\int_{[C_0]} c_1(T_X)&\le \int_\beta c_1(T_X)-n\cdot s(X)\\
&\le (r-s(X))n+O(1)
\end{align*}
by \eqref{dim_constraint_weak}, where again the term $O(1)$ depends only on $g$ and $X$ and not on $C,f,\beta$. Since $X$ satisfies property $(\star\star)_g$, we conclude $\dim(\cM_\Gamma)<\dim(\barM_{g,n}\times X^n)$ for $n$ sufficiently large, as desired.
\end{proof}

Under stronger assumptions as in Proposition \ref{rational_tails_simple}, we now rule out stable maps at the boundary in a general fiber of $\tau$ with arbitrary topology.

\begin{prop}\label{rational_tails_trees}
Suppose $X$ satisfies property $(\star\star)_g$ and $n$ is sufficiently large (depending on $X$ and $g$). Let
$$\cM_\Gamma\subset \barM_{g,n}(X,\beta)$$
be a locally closed boundary stratum consisting of stable maps $f:C\to X$ such that $C$ is the union of a nonsingular genus $g$ curve $C_0$, disjoint trees of nonsingular rational curves $T_1,\ldots,T_{n-m},T'_1,T'_2,\ldots,T'_k$, such that $x_i\in T_{i}$ for $i=1,2,\ldots,n-m$ and $x_i\in C_0$ for $i=n-m+1,\ldots,n$. (See Figure \ref{Fig:domain_trees}.)

Suppose the virtual dimension of $\barM_{g,n}(X,\beta)$ is equal to the dimension of $\barM_{g,n}\times X^n$,
\begin{equation*}
r(1-g)+\int_\beta c_1(T_X)= rn\, .
\end{equation*}
Assume further that at least one of the following two conditions hold:
\begin{enumerate}
\item[(i)] $f^{*}T_X$ is globally generated for every $f:\bP^1\to X$, 
\item[(ii)] $s(X)+t(X)\ge r+1$.
\end{enumerate}
Then, $\cM_{\Gamma}$ fails to dominate $\barM_{g,n}\times X^n$.
\end{prop}

\begin{figure}[!htb]
     \includegraphics[width=.75\linewidth]{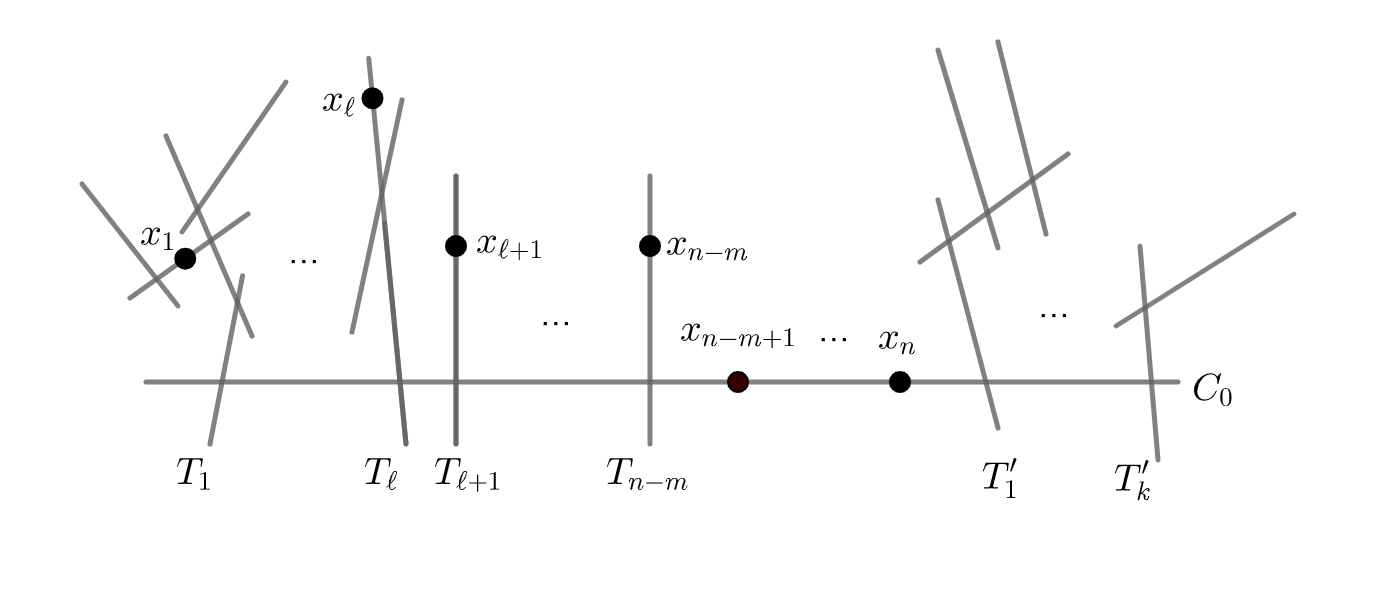}
     \caption{Arbitrary domain of a stable map whose stable contraction is nonsingular}
     \label{Fig:domain_trees}
\end{figure}

\begin{proof}
For (i), we may apply the  proof of Proposition \ref{rational_tails_simple} using the stronger Lemma \ref{unirational_expdim}(b) instead of Lemma \ref{unirational_expdim}(a) to conclude.

Consider now (ii). Without loss of generality, let $\ell$ be such that among the trees $T_1,\ldots,T_{n-m}$, those containing more than one rational curve are $T_1,\ldots,T_\ell$. 
We have
\begin{equation}\label{lowerbound_l}
\ell\le\frac{1}{(s(X)+t(X))} \int_\beta c_1(T_X)\le\frac{rn}{r+1}+O(1)\, .
\end{equation}
where the constant term $O(1)$ depends only on $g$ and $X$.

We now apply Proposition \ref{rational_tails_simple} to the stable maps $\widehat{f}:\widehat{C}\to X$ obtained by deleting $T_1,\ldots,T_\ell$, $x_1,\ldots,x_\ell$, and $T'_1,\ldots,T'_k$ from $C$. 
Since $X$ is Fano, $\int_{T_i'} c_1(T_X) \geq 0$.
We have
\begin{align*}
\int_{[\widehat{C}]} c_1(T_X) &\le \int_\beta c_1(T_X) - \ell(s(X)+t(X)) \\
&= r(n+g-1) - \ell(s(X)+t(X))\\
&< r(n-\ell+g-1)\, ,
\end{align*}
so Proposition \ref{rational_tails_simple} indeed applies. We find that when $n-\ell$ is sufficiently large (which occurs whenever $n$ is sufficiently large, by \eqref{lowerbound_l}), the space of stable maps $\widehat{C}\to X$ does not dominate $\barM_{g,n-\ell}\times X^{n-\ell}$. In particular, $\cM_\Gamma$ fails to dominate $\barM_{g,n}\times X^n$.
\end{proof}

\section{Asymptotic enumerativity}

We can now state our most general result
concerning the enumerativity of
virtual Tevelev degrees.

\begin{thm}\label{when_enum}
Suppose that $X$ satisfies property $(\star\star)_g$ and additionally that one of the following
two conditions hold:
\begin{enumerate}
\item[(i)] $f^{*}T_X$ is globally generated for every $f:\bP^1\to X$, 
\item[(ii)] $s(X)+t(X)\ge r+1$.
\end{enumerate}
Then, $\vTev^X_{g,\beta}$ is enumerative whenever $\int_\beta c_1(T_X)$ is sufficiently large (depending only on $X$ and $g$).
\end{thm}

\begin{proof}
The requirement that $\int_\beta c_1(T_X)$ be sufficiently large is equivalent to the requirement that $n$ be sufficiently large.

Proposition \ref{rational_tails_trees} shows that the general fiber of $$\tau:\barM_{g,n}(X,\beta)\to\barM_{g,n}\times X^n$$ 
consists only of stable maps with nonsingular domains. Proposition \ref{transversality_smooth} then shows that the general fiber consists of finitely many reduced points, the number of which is equal to $\Tev^X_{g,\beta}$.
\end{proof}

\begin{proof}[Proof of Theorem \ref{thm_g/p}]
Every $X=G/P$ satisfies property $(\star)_g$ and hence property $(\star\star)_g$, and also satisfies property (i) above. The result is then immediate from Theorem \ref{when_enum}.
\end{proof}

\begin{ex}
{\em Every $X$ satisfying $s(X)\ge r+1$ satisfies $(\star\star)_g$ and (ii), so for any such $X$, virtual Tevelev degrees are enumerative for $\beta$ sufficiently large.}
\end{ex}

In order to prove Theorem \ref{thm_hypersurface}, we show that if $X$ is a complete intersection of low degree in a Fano variety $Y$ satisfying $(\star)_g$, then $X$ satisfies $(\star\star)_g$. For simplicity, we assume the Picard rank
of $Y$ is 1.

\begin{prop}\label{complete_intersection}
Let $Y$ be a nonsingular projective Fano variety of Picard rank 1 with positive generator $\cO(1)\in\Pic(Y)$. Suppose further that $Y$ satisfies property $(\star)_g$.

Let $X\subset Y$ is a nonsingular complete intersection of dimension $r$ and degree $(e_1,\ldots,e_k)$ (with the degrees computed with respect to $\cO(1)$). 
%Let $\beta\in H_2(X,\bZ)$ be a curve class and define
%\begin{equation*}
%a=\frac{\int_X \beta\cdot c_1(T_X)}{\int_Y \beta\cdot \cO(1)}\in\bQ_{>0}.
%\end{equation*}
%
%Suppose further that 
%\begin{equation*}
%(r-a\cdot s(X))\sum_{i=1}^{k}e_i<a.
%\end{equation*}
Suppose, for all curve classes $\beta\in H_2(X,\bZ)$, that
\begin{equation}\label{complete_intersection_ineq}
\frac{\int_{\beta} c_1(T_X)}{\int_\beta c_1(\cO(1))}>(r-s(X))\sum_{i=1}^k e_i\, .
\end{equation}
Then, $X$ satisfies property $(\star\star)_g$.
\end{prop}

\begin{rem}
If $\dim(X)\ge3$, the Lefschetz hyperplane theorem guarantees that the left hand side of \eqref{complete_intersection_ineq} only needs to be computed for one non-zero effective class $\beta$.
\end{rem}

\begin{proof}[Proof of Proposition \ref{complete_intersection}]
Let $C$ be nonsingular, and
let $f:C\to X$ be a stable map.
Let $i:X\to Y$ be the inclusion. From the exact sequence
\begin{equation*}
H^0(C,f^{*}N_{X/Y})\to H^1(C,f^{*}T_X)\to H^1(C,f^{*}i^{*}T_Y)\, ,
\end{equation*}
we have
\begin{align*}
h^1(C,f^{*}T_X)&\le h^0(C,f^{*}N_{X/Y})+h^1(C,f^{*}i^{*}T_Y)\\
&= h^0(C,f^{*}(\cO(e_1)\oplus\cdots\cO(e_k))) + O(1)\\
&=\left(\int_ \beta c_1(\cO(1))\right)
\sum_{i=1}^k e_i + O(1)\, ,
\end{align*}
where the constant upper bound $O(1)$ depends only on $g$ and $Y$. This implies the claim.
\end{proof}

\begin{proof}[Proof of Theorem \ref{thm_hypersurface}]
In Proposition \ref{complete_intersection}, we take $r\ge3$, $Y=\bP^{r+1}$, and $X$ to be a hypersurface of degree $e=e_1\le r$. Then, $$s(X)\geq t(X)=\int_{[L]} c_1(T_X)=r+2-e\, ,$$ where $[L]$ is the class of a line. If $r>(e+1)(e-2)$, 
then Proposition \ref{complete_intersection} applies and
 $X$ satisfies property $(\star\star)_g$. Moreover, we have $$s(X)+t(X)\geq 2(r+2-e)\ge r+1\, ,$$ so we are done by Theorem \ref{when_enum}.
\end{proof}

Hypersurfaces $X_e\subset \bP^{r+1}$ are homogeneous spaces
for $e=1$ and $2$. By Theorem \ref{thm_g/p}, the virtual Tevelev degrees
are enumerative for all curve classes of sufficiently high degree (depending upon the genus). While a direct approach to
the geometric Tevelev degrees is explained in the $e=1$ 
case in \cite{fl}, how to directly calculate the geometric Tevelev degrees
for quadrics in the asympototic range
(and to match the
the quadric formula of Theorem \ref{quadric}) is an interesting
question in projective geometry.

For cubic hypersurfaces $X_3 \subset \bP^{r+1}$, the virtual
Tevelev degrees are calculated in \cite{bp}
for all $r\geq 3$ by the simple
formula of Theorem \ref{virtualtev_hypersurface}. By Theorem \ref{thm_hypersurface}, the virtual Tevelev
degrees are enumerative for all curves classes of sufficiently
high degree (depending upon the genus) for all $r\geq 5$.

\begin{question} \label{qqq8}
Find a direct calculation of the
geometric Tevelev degrees in the asymptotic range for hypersurfaces  via the projective geometry of curves.
\end{question}

In fact, a geometric derivation of the formula of Theorem \ref{virtualtev_hypersurface} has been recently given in \cite{lian_hypersurface}, but the case of quadrics remains open.

\section{Refined results in the genus 0 case}
Our arguments yield stronger results for the
enumerativity of virtual Tevelev degrees in the genus 0 case: for certain $X$, \emph{all} well-posed 
virtual Tevelev degrees are enumerative, with no assumption on the positivity of $\beta$. The improvements will be needed
for the application to curves on hypersurfaces
in positive characteristic in Section \ref{ppp}.

Let
$X$ be a projective Fano variety of dimension $r$.
We first introduce the following more precise version of the property $(\star\star)_0$. 
\begin{defn}
We say that $X$ has property $(\star\star)'_0$ if
\begin{equation}\label{g=0_precise_condition}
    (r-s(X))\cdot h^1(\bP^1,f^{*}T_X)<r+\int_{\beta}c_1(T_X)
\end{equation}
for every  curve class $\beta$ and for every $[f:\bP^1\to X]\in M_0(X,\beta)$.
\end{defn}

\begin{rem}
Property $(\star\star)'_0$ is immediate when $s(X)\ge r$ or when $h^1(\bP^1,f^{*}T_X)=0$ for all maps $f:C\to X$ in class $\beta$.
\end{rem}

We now have the following refined version of Proposition \ref{rational_tails_simple}.

\begin{prop}\label{rational_tails_simple_g=0}
Suppose that $g=0$, and that $X$ satisfies $(\star\star)'_0$. As in Proposition \ref{rational_tails_simple}, 
let $$M_{\Gamma}\subset \overline{M}_{0,n}(X,\beta)$$ be a locally closed boundary stratum with topology as in Figure \ref{Fig:domain_tails}, where we require the spine $C_0$ now to be rational.

Suppose further that the virtual dimension of $\overline{M}_{0,n}(X,\beta)$ is at most the dimension of $\overline{M}_{0,n}\times X^n$,
\begin{equation}\label{dim_constraint_weak}
r+\int_\beta  c_1(T_X)\le rn\, ,
\end{equation}
and furthermore that, if equality holds, then $n-m>0$.

Then, $\dim(M_\Gamma)<\dim(\overline{M}_{0,n}\times X^n)$. In particular, $M_\Gamma$ fails to dominate $\overline{M}_{0,n}\times X^n$.
\end{prop}

\begin{proof}
We employ the same proof as in Proposition \ref{rational_tails_simple}. We immediately reduce to the case $m=0$. 
By Lemma \ref{unirational_expdim}(a), we have
\begin{equation*}
\dim(M_{\Gamma})\le\dim(\overline{M}_{0,n}\times X^n)+h^1(C_0,T_X)-n.
\end{equation*}
Furthermore, we have
\begin{align*}
\int_{[C_0]} c_1(T_X)&\le \int_\beta c_1(T_X)-n\cdot s(X)\\
&\le (r-s(X))n-r
\end{align*}
by \eqref{dim_constraint_weak}.
Since $X$ satisfies property $(\star\star)'_0$, we conclude $\dim(M_\Gamma)<\dim(\overline{M}_{0,n}\times X^n)$ for \emph{all} $n$, as desired.
\end{proof}

Next, we have the following analog of Proposition \ref{rational_tails_trees}:

\begin{prop}\label{rational_tails_trees_g=0}
Suppose that $g=0$, and that $X$ satisfies property $(\star\star)'_0$.
Let
$$M_\Gamma\subset \overline{M}_{0,n}(X,\beta)$$
be any locally closed boundary stratum as in Proposition \ref{rational_tails_trees}.

Suppose the virtual dimension of $\overline{M}_{0,n}(X,\beta)$ is equal to the dimension of $\overline{M}_{0,n}\times X^n$,
\begin{equation*}
r+\int_\beta c_1(T_X)= rn\, .
\end{equation*}
Assume further that at least one the conditions (i) or (ii) of Proposition \ref{rational_tails_trees} holds.
Then, $M_{\Gamma}$ fails to dominate $\overline{M}_{0,n}\times X^n$.
\end{prop}

\begin{proof}
The proof of Proposition \ref{rational_tails_trees} goes through immediately; note that no inequality on $\ell$ is needed because Proposition \ref{rational_tails_simple_g=0} holds for $n$ arbitrary.
\end{proof}

\begin{cor}\label{when_enum_g=0}
Suppose that $g=0$, and that $X$ satisfies property $(\star\star)'_0$, If condition (i) or (ii) of Theorem \ref{when_enum} holds,
then, $\vTev^X_{0,\beta}$ is enumerative.
\end{cor}

\begin{proof}
Immediate from Propositions \ref{transversality_smooth} and \ref{rational_tails_trees_g=0}.
\end{proof}

\begin{cor}\label{hyp_g=0}
Suppose $X_e\subset\bP^{r+1}$ is a nonsingular hypersurface of degree $e$ and dimension
\begin{equation*}
    r>(e+1)(e-2).
\end{equation*}
Then, all genus 0 virtual Tevelev degrees $\vTev^{X_e}_{0,\beta}$ are enumerative.
\end{cor}

\begin{proof}
We follow the proofs of Proposition \ref{complete_intersection} and Theorem \ref{thm_hypersurface}. If $g=0$ and $Y=\bP^{r+1}$ (more generally, if $Y=G/P$), then the $O(1)$ term in the upper bound on $h^1(C,f^{*}T_X)$ goes away:
$$h^1(C,f^{*}T_X) \leq \left(\int_ \beta c_1(\cO(1))\right) e +1 \, .$$
We find that if $r>(e+1)(e-2)$, then $X_e$ satisfies $(\star\star)'_0$, so we conclude by Corollary \ref{when_enum_g=0}.
\end{proof}

% \begin{equation}\label{complete_intersection_ineq_g=0}
% \frac{r+\int_{\beta} c_1(T_X)}{\int_\beta c_1(\cO(1))}>(r-s(X))\sum_{i=1}^k e_i\, .
% \end{equation}

\section{Very free rational curves in characteristic $p$}\label{very_free_sec}
Let $k$ be an algebraically closed field of arbitrary characteristic. Let $X$ be a nonsingular projective variety over $k$.
A morphism $$f:\bP^1\to X$$ is a \textbf{very free rational curve} on $X$ if $f^{*}T_X$ is ample. The variety $X$ is \textbf{separably rationally connected} if it has a very free rational curve. The condition is equivalent by \cite[Theorem 3.7]{kollar_book} to the existence of a curve class $\beta$ on $X$ such that the evaluation map $$\ev:M_{0,2}(X,\beta)\to X\times X$$ is dominant and separable on at least one component of the space $M_{0,2}(X,\beta)$.
In characteristic zero, the separability of $\ev$ is immediate. It is known that Fano varieties are (separably) rationally connected in characteristic zero \cite{kmm}, and are conjectured to be so in arbitrary characteristic.

Our results on the enumerativity of Tevelev degrees in characteristic zero imply the existence of very free curves on certain Fano hypersurfaces in characteristic $p$.

\begin{thm}\label{very_free_thm}
Let $k$ be an algebraically closed field of characteristic $p>0$, and let $X_e\subset\bP^{r+1}_{k}$ be a nonsingular hypersurface of degree $e\ge3$. Fix integers $d,n\ge3$ satisfying
\begin{equation*}
    d=(n-1)\cdot\frac{r}{r+2-e}.
\end{equation*}
Assume that:
\begin{enumerate}
    \item $r>(e+1)(e-2)$, and
    \item $p>e$.
\end{enumerate}
Then, $X_e$ contains a very free rational curve of degree at most $d$ (where the degree is measured in the ambient projective space). In particular, $X_e$ is separably rationally connected.
\end{thm}

On a \emph{general} hypersurface $X_e\subset\bP^{r+1}_{k}$, very free curves of degree $r+1$ were constructed by Zhu \cite{zhu} without assuming $(i)$ or $(ii)$. Very free curves on general complete intersections were similarly constructed by Chen-Zhu in \cite{cz}. When $\gcd(r+2-e,r)>1$, our result gives very free curves of lower degree for \emph{arbitrary} hypersurfaces satisfying assumptions $(i)$ and $(ii)$.

The separable rational connectivity for \emph{arbitrary} Fano hypersurfaces $X_e$ (and more generally, for arbitrary Fano complete intersections) was proven only assuming $(ii)$ when $e<r+1$, and with an additional divisibility condition when $e=r+1$, by Starr-Tian-Zong in \cite{STZ}, but there the very free curves constructed are of unspecified degree.

\begin{proof}[Proof of Theorem \ref{very_free_thm}]
Let $R$ be a discrete valuation ring with fraction field $K$ of characteristic 0 and residue field isomorphic to $k$. Let $\mathfrak{X}\subset\bP^{r+1}_R$ be a smooth hypersurface of degree $e$ over $\Spec(R)$ with special fiber isomorphic to $X_e$.

Let $\overline{M}_{0,n}(\mathfrak{X},d[L])_R$ be the relative moduli space (over $R$) of stable maps of degree $d$ (as computed against the hyperplane class) in $\mathfrak{X}$, and let $$\pi_R:\cC_{0,n}(\mathfrak{X},d[L])_R\to \overline{M}_{0,n}(\mathfrak{X},d[L])_R$$ be the universal family. Let $$h_R=\tau_R:\overline{M}_{0,n}(\mathfrak{X},d[L])_R\to(\overline{M}_{0,n})_R\times \mathfrak{X}^n$$ be the forgetful map.

Combining Theorem \ref{virtualtev_hypersurface} and Corollary \ref{hyp_g=0}, the degree of the map (in characteristic 0)
$$h_{\overline{K}}:\overline{M}_{0,n}(X_{\overline{K}},d[L])\to \overline{M}_{0,n}\times X_{\overline{K}}^n$$ is equal to $$\vTev^{X_{\overline{K}}}_{0,d} = ((e-1)!)^{n} \cdot  e^{(d-n)e+1},$$ which in particular is not divisible by $p$.

Applying \cite[Corollary 3.3]{STZ}, we conclude that the special fiber $\overline{M}_{0,n}\times X^n$ has a free curve $\bP^1_k\to \overline{M}_{0,n}\times X^n$, and upon projection we obtain a very free curve of degree at most $d$ on $X$.
\end{proof}

\end{document}